\documentclass[12pt]{amsart}

\usepackage{amsmath}
\usepackage{amsfonts}
\usepackage{amssymb}
\usepackage{amsthm}

\usepackage{upgreek}
\usepackage{bigdelim}
\usepackage{multirow}
\usepackage{hhline}
\usepackage{enumerate}
\usepackage{comment}
\usepackage{amsthm,verbatim}
\usepackage{latexsym}
\usepackage{fancyhdr}

\newcommand{\Z}{\mathbb{Z}}
\newcommand{\N}{\mathbb{N}}
\newcommand{\R}{\mathbb{R}}

\renewcommand{\P}{\mathbb{P}}
\renewcommand{\epsilon}{\varepsilon}
\renewcommand{\Lambda}{\Uplambda}
\newcommand{\ve}{\varepsilon}
\newtheorem{theorem}{Theorem}[section]
\newtheorem{lemma}[theorem]{Lemma}

\newtheorem{cor}[theorem]{Corollary}
\newtheorem{prop}[theorem]{Proposition}

\newtheorem{definition}{Definition}

\newtheorem{remark}{Remark}
\renewcommand{\mod}[1]{{\ifmmode\text{\rm\ (mod~$#1$)}\else\discretionary{}{}{\hbox{ }}\rm(mod~$#1$)\fi}}

\begin{document}
\title{Ergodicity and Conservativity of products of infinite  transformations and their inverses}
\date{\today}
\author[Clancy]{Julien Clancy}
\address[Julien Clancy]{Yale University, New Haven, CT 06520, USA}
\email{julien.clancy@yale.edu}

\author[Friedberg]{Rina Friedberg}
\address[Rina Friedberg]{University of Chicago, Chicago, IL 60637, USA}
\email{rinafriedberg@uchicago.edu}

\author[Kasmalkar]{Indraneel Kasmalkar}
\address[Indraneel Kasmalkar]{University of California, Berkeley, Berkeley, CA 94720, USA}
\email{indraneelk@berkeley.edu}

\author[Loh]{Isaac Loh}
\address[Isaac Loh]{Williams College, Williamstown, MA 01267, USA}
\email{il2@williams.edu}

\author[P\u{a}durariu]{Tudor P\u{a}durariu}
\address[Tudor P\u{a}durariu]{University of California, Los Angeles, CA 90095-1555, US }
\email {tudor\_pad@yahoo.com}

\author[Silva]{Cesar E. Silva}
\address[Cesar E. Silva]{Department of Mathematics\\
     Williams College \\ Williamstown, MA 01267, USA}
\email{csilva@williams.edu}

\author[Vasudevan]{Sahana Vasudevan}
\address[Sahana Vasudevan]{Harvard University, Cambridge, MA 02138, USA}
\email{svasudevan@college.harvard.edu}


\pagestyle{fancy}
\fancyhead{}
\fancyhead[CO]{Ergodicity of Products in Infinite Measure}
\fancyhead[CE]{Clancy, Friedberg, Kasmalkar, Loh, P\u{a}durariu, Silva, and Vasudevan}

\begin{abstract}
We construct a class of   rank-one infinite measure-preserving transformations  such  that for each transformation $T$ in the class,  the cartesian product $T\times T$ of the transformation with itself is ergodic, but the product $T\times T^{-1}$  of the transformation with its inverse is not ergodic. We also  prove  that the product of any  rank-one transformation  with its inverse  is conservative, while there are infinite measure-preserving conservative ergodic Markov shifts whose product with their inverse is not conservative.
\end{abstract}

\maketitle

\section{Introduction}

The notion of weak mixing for finite measure-preserving transformations has many equivalent characterizations. Several of these characterizations, however, do not remain equivalent in the infinite measure-preserving case. The first examples showing that some of the properties are different  in the  infinite measure case were given by  Kakutani and Parry  \cite{KaPa63}, who constructed, for each positive integer $k$,  an infinite measure-preserving Markov shift $T$ such that the $k$-fold cartesian product of $T$ with itself is ergodic but its $k+1$-fold product is not (such a transformation is said to have {\bf ergodic index $k$}). Later,  Adams, Friedman and Silva \cite{AdFrSi01} constructed a rank-one infinite measure-preserving transformation $T$ with {\bf infinite ergodic index}  (i.e., all  finite cartesian products with itself are ergodic) but such that $T\times T^2$ is not conservative, hence not ergodic. Bergelson then asked if there existed  an example of a transformation $T$ of infinite ergodic index  but such that  $T\times T^{-1}$ is not ergodic.  This question appears as Problem 10 in \cite{Da08}. For the history and other examples, the reader may refer to \cite{DaSi09}; more recently though, ergodic index $k$ transformations have been constructed in rank-one in \cite{AdSi14}.   In this paper we partially answer Bergelson's  question by constructing an infinite measure-preserving rank-one transformation $T$ such that $T\times T$ is ergodic, but $T\times T^{-1}$ is not ergodic (Theorem~\ref{T:ieinotcons}). We also prove that for all rank-one transformations $T$, the transformation $T^n\times T^{-n}$ is conservative (Theorem~\ref{T:invcons}) for all $n\neq 0$ (the main result is that $T\times T^{-1}$ is conservative, as it is known that composition powers of a conservative transformations are conservative, see e.g. \cite[Corollary 1.1.4]{Aa97}), while this is not the case in general (Corollary~\ref{C:tt-1notcons}). In this context we note that it was already known that  there exist rank-one transformations $T$ such that $T\times T$ is not conservative
\cite{AdFrSi97}. Also, whenever $T$ is a {\bf rigid} transformation (i.e., there is  an increasing  sequence $\{n_i\}$ such that the limit the measure of
$T^{n_i}(A)\bigtriangleup A$ tends to $0$ for all sets $A$ of finite measure) one can verify that $T\times T^{-1}$ is conservative, and as the class of rigid transformations is generic 
in the group of invertible infinite measure-preserving transformations of a Lebesgue space under the weak topology \cite{AgSi02}, it follows that the property of $T\times T^{-1}$ being conservative is a generic property; this fact also follows from Theorem~\ref{T:ieinotcons} and the fact that infinite measure-preserving rank-ones are generic \cite{BSSSW15}. As we show later, however, there are other transformations, in particular conservative ergodic Markov shifts, where the product $T\times T^{-1}$ is not conservative (Corollary~\ref{C:tt-1notcons}). A consequence of the properties of our rank-one  examples in Theorem~\ref{T:ieinotcons} is that these transformations are not isomorphic to their inverse. Also, it follows from Theorem~\ref{T:invcons} that if a rank-one transformation $T$ satisfies that $T\times T$ is not conservative, then $T$ is not isomorphic to its inverse.

The methods that we use are combinatorial and probabilistic in nature. Propositions \ref{P:kproduct} and \ref{P:arbproduct} use the notion of descendants, as introduced in  \cite{DGPS15}, to turn the dynamics of the rank-one system into combinatorial characterizations.

We let $(X, \mathcal{B}, \mu)$ denote a Lebesgue measurable subset of the real line with Borel measureable sets $\mathcal{B}$, and consider $T: X \to X$ an invertible measure-preserving  transformation; we are interested in the case when $X$ is of infinite measure.  The transformation  $T$ is {\bf ergodic} if whenever  $T^{-1}(A) = A$, then $\mu(A) = 0$ or $\mu(A^c) = 0$, and {\bf conservative} if $A\subset \bigcup_{n=1}^\infty T^{-n}(A)  \mod \mu$. $T$ is invertible and a standard proof shows that the Lebesgue measure on $\R$ is non atomic, so when $T$ is ergodic, it is conservative. 

We review rank-one cutting-and-stacking transformations. A \textbf{Rokhlin column} or {\bf column} $C$ is an ordered finite collection of pairwise disjoint intervals (called the $\mathbf{levels}$ of $C$) in $\mathbb{R}$, each of the same measure. We think of the levels in a column as being stacked on top of each other, so that the $(j + 1)$-st level is directly above the $j$-th level. Every column $C = \{I_j\}$ is associated with a natural column map $T_C$ sending each point in $I_j$ to the point directly above it in $I_{j+1}$ (note that $T_C$ is undefined on the top level of $C$). A $\mathbf{rank}$-$\mathbf{one}$ $\mathbf{cutting}$-$\mathbf{and}$-$\mathbf{stacking}$ construction for $T$ consists of a sequence of columns $C_n$ such that:

\begin{enumerate}
	\item The first column $C_0$ consists only of the unit interval.
	\item Each column $C_{n+1}$ is obtained from $C_n$ by cutting $C_n$ into $r_n \geq 2$ subcolumns of equal width, adding any number $s_{n,k}$ of new levels (called $\mathbf{spacers}$) above the $k$th subcolumn, $k \in \{0,r_n - 1\}$, and stacking every subcolumn under the subcolumn to its right. In this way, $C_{n+1}$ consists of $r_n$ copies of $C_n$, possibly separated by spacers.
	\item $X = \bigcup_n C_n$.
\end{enumerate}

\noindent
Observing that $T_{C_{n+1}}$ agrees with $T_{C_n}$ everywhere that $T_{C_n}$ is defined, we then take $T$ to be the pointwise limit of $T_{C_n}$ as $n \rightarrow \infty$. A transformation constructed with these cutting and stacking techniques is rank-one, and in practice we often refer to cutting and stacking transformations as rank-one transformations. For further details on this class of transformations, the reader may refer to \cite{Si08} and \cite{BSSSW15}.

Given any level $I$ from $C_m$ and any column $C_n$ of $T$ with $m \leq n$, we define the $\mathbf{descendants}$ of $I$ in $C_n$ to be the collection of levels in $C_n$ whose disjoint union is $I$. We denote this set by $D(I, n)$. By abuse of notation (and not to complicate the notation further), we will also use $D(I, n)$ to refer to the heights of the descendants of $I$ in $C_n$.

For $j \ge 0$, let $h_j$ denote the order of $C_j$, and write $h_{j,k} = h_j + s_{j,k}$. 
Suppose that $I$ is a level in $C_i$ of height $h(I)$, where the heights in the column are $0$-indexed. Then $I$ splits into $r_i$ levels in $C_{i+1}$ of heights
\begin{equation*}
	\{h(I)\} \bigcup \{h(I) + \sum_{k = 0}^i h_{j, k} \mid 0 \leq i < r_j - 1\}
\end{equation*}
Letting 
\begin{equation} \label{E:hjeqn}
H_j = \{0\} \bigcup \left\{\sum_{k = 0}^i h_{j, k} \mid 0 \leq i < r_j - 1\right\},
\end{equation}
it follows inductively that
\begin{equation}\label{E:descendants}
	D(I, n) = h(I) + H_i \oplus H_{i+1} \oplus \cdots \oplus H_{n - 1}
\end{equation}
we call the set $H_k$ the \textbf{height set} of $T$ at the $k^\text{th}$ stage. 

Instead of describing a rank-one transformation by cutting and spacer parameters, we can describe it by specifying its descendant sets. For instance, given $D([0, 1], n)$ for every $n$, we have complete information on the distribution of spacer levels and non-spacer levels (descendants of $I$) in $C_n$ below $\max D([0,1],n)$ for every $n \in \N$, which symbolically describes $T$ while bypassing the traditional cutting and stacking notation. On the other hand, if one wishes to construct a rank-one transformation, then one needs only to specify sets $H_k \subset \N$ for $k \ge 0$ and define $D([0, 1], n)$ as above, that is, $D([0, 1], n) = H_0 \oplus \ldots \oplus H_{n-1}$. The only compatibility restrictions that follow from (\ref{E:hjeqn}) are that $0 \in H_k$ for all $k$, and that any two elements of $H_k$ are at least $h_{k-1}$ apart, where $h_{k-1}$ is the height of column $C_{k-1}$.

\subsection{Acknowledgements}
This paper was based on research done by the Ergodic Theory groups of the 2012 and 2014 SMALL Undergraduate Research Project at Williams College. Support for this project was provided by the National Science Foundation REU Grant  
DMS-0850577 
and DMS - 1347804 and the Bronfman Science Center of Williams College.  We also thank the referee for comments that improved our exposition.

\section{Preliminaries}\label{S:prelim}

Throughout this paper, we will let $T^{(k)} = T \times \cdots \times T$, and $U = T \times T^{-1}$.  For positive measure sets $A,B \subset X$ and any $\ve \in (0,1)$, we take $A \subset_\ve B$ to mean that $\mu(A \cap B) > (1-\ve)\mu(A)$. The operator $| \cdot |$ denotes the order of any subset of the integers.

In this section, we develop some techniques whereby cutting and stacking transformations can be characterized by their integer properties. We find that many of the ergodic properties of products of rank-one transformations can be deduced from the heights of pieces of the base $I$ of column $C_i$ in subsequent columns $C_j$, $j > i$. These heights can be inferred from the descendant set $D(I,j)$. 
First, we have a standard lemma which provides a sufficient condition for the ergodicity of $T\times T$, when $T$ is a rank-one cutting and stacking transformation. This is similar to Lemma 2.4 from \cite{Danilenko2001}, but in the case of integer actions. 

We say that a map $\tau: A\rightarrow X$ belongs to the \textit{full grouppoid} of $T$ and write $\tau \in [[T]]$ if $\tau$ is one-to-one and $\tau(x) \in \big\{T^n x: \, n \in \Z\big\}$ for all $x\in A$.

\begin{lemma}\label{L:semiringce}
	Let $T$ be a rank-one measure-preserving transformation on measure space $Y$. Let $X = Y \times Y$, and $\mathcal{D}$ be the sufficient semiring of rectangles in $X$ of the form $R_1 \times R_2$, where $R_1$ and $R_2$ are levels of some column of $T$.
Then $T\times T$ is ergodic on $X$ if for every $A,B \in \mathcal{D}$ of the form $A = I \times I$ for $I$ the base of column $C_i$, $i\in \N$, and $B = I \times  T^{b} I$, where  $0\le b< h_i$, there exists a map $\tau \in [[T \times T]]$ satisfying:
\begin{align*}
&D(\tau) \subset A \text{ and }R(\tau) \subset B,\\
&\text{ and }\mu(D(\tau)) \ge \delta \mu(A),\\
&\text{ and }\frac{d\mu\circ \tau}{d\mu}(v) \ge\beta \text{ for all } v\in D(\gamma).
\end{align*}
where $\delta$ and $\beta$ are positive absolute constants, and $D$ and $R$ denote the domain and range of a map, respectively. 
\end{lemma}

\begin{proof}
	Let $E$ and $F$ be two sets of positive measure in $X$. Because $\mathcal{D}$ is a sufficient semiring, we can find rectangles $A'$ and $B'$ in $\mathcal{D}$ such that $A'$ and $B'$ are more than $1-\frac{1}{32}$ full of $E$ and $F$, respectively. By taking $i$ sufficiently high, we can assume that $A'$ and $B'$ are rectangles in $C_i\times C_i$. So we have $A' = T^{a_0'} I \times T_{k-1}^{a_1'} I$ and $B' = T^{b_0'} I \times T^{b_1'} I$, where $a_0',a_1',b_0',b_1'\in \{0,...,h_i-1\}$. By the Double Approximation Lemma\footnote{For a proof of this well-known lemma, see \cite{BoFiMaSi01}.}, for some $j > i$, more than $1-\frac{1}{32}$ of the subrectangles of $A'$ and $B'$ in $C_j\times C_{j}$ must be more than $\frac{3}{4}$-full of $E$ and $F$, respectively. 
	
	We claim that we can find $j$-subrectangles with sides in $C_j$ which we denote $A\subset A'$ and $B\subset B'$ that are more than $\left(1-\frac{\beta \delta}{2}\right)$-full of $E$ and $F$ respectively, with the levels of sides of $A$ at or below the corresponding sides of $B$. Let $J$ be the base of $C_j$. Then all of the $j$-subrectangles $A\subset A'$  and $B\subset B'$ are of the form
\begin{align*}
A &= T^{a_0' + a_0''} J \times T^{a_1' + a_1''} J \\
B &= T^{b_0' + b_0''} J \times T^{b_1' + b_1''} B \\.
\end{align*}
where 
\[a_0'',a_1'',d_0'',d_1'' \in H_i \oplus \dots \oplus H_{j-1} = D(I,j).\]

Set $a_\ell = a_\ell' + a_\ell''$ and $d_\ell = d_\ell' + d_\ell''$ for $\ell = 0,1$. Recall that $\min\Big\{|x_1-x_2|:\, x_1,x_2\in D(I,j) \text{ and }x_1\ne x_2\Big\}\ge h_i$, so if $a_\ell''< d_\ell''$ for any $\ell = 0,1$, we should have $a_\ell < d_\ell$. For either setting of $\ell$, the total number of pairs $(a_\ell'',d_\ell'')$ such that $a_\ell'' < d_\ell''$ is bounded below by, say $\frac{1}{4}|D(I,j)|^2$, so the total number of elements in $D(I,j)^2 \times D(I,j)^2$ of the form $\Big((a_0,a_1),(d_0,d_1)\Big)$ which have $a_0 < d_0$ and $a_1< d_1$ is bounded below by $\frac{1}{16} |D(I,j)|^4$. Each pair denotes two $j$-subrectangles, one of $A$ and the other of $B$. But more than some fraction $1-\frac{1}{32}$ of the subrectangles of $A$ are more than $\left(1-\frac{\beta \delta}{2}\right)$-full of $E$, and the same is true with subrectangles of $B$ and $F$. So a fraction larger than $1-\frac{1}{16}$ of the possible pairs in $D(I,j)^2 \times D(I,j)^2$ denote in order a subrectangle of $A$ which is more than $\left(1-\frac{\beta \delta}{2}\right)$-full of $E$ and a subrectangle of $B$ which is more than $\left(1-\frac{\beta \delta}{2}\right)$-full of $F$. Hence, there is a subrectangle $B\subset B'$ which is more than $\left(1-\frac{\beta \delta}{2}\right)$-full of $F$ that has both of its sides (indexed by heights $d_0,d_1$) above the corresponding sides of a subrectangle $A$ which is more than $\left(1-\frac{\beta \delta}{2}\right)$-full of $E$. 

Suppose first that $d_0 - a_0 \le d_1 - a_1$. Then $d_1 - a_1 - d_0 + a_0 \in \{0,...,h_j-1\}$. Now define the map $\gamma$ by 
\[\gamma: = \Big( T^{d_0} \times T^{a_1 + d_0 - a_0}\Big) \circ \tau \circ \Big( T^{-a_0} \times T^{- a_1}\Big),\]
where $\tau \in [[T\times T]]$ is the map with $D(\tau) \subset J \times J$ and $R(\tau) \subset J \times T^{d_1-a_1 - d_0 + a_0} J$, and $\gamma$ is constructed around $\tau$ as a map from $A$ to $B$. By supposition, $\mu(D(\gamma)) \ge \delta\mu(A)$, so $\mu\big(D(\gamma) \cap E\big) > \left(\delta - \frac{\beta\delta}{2}\right) \, \mu(A) \ge \frac{\delta}{2} \, \mu(A)$. Thus, 
\[\mu\left(\gamma\big(D(\gamma) \cap E) \big) \cap B \cap F\right) > \beta \mu\big(D(\gamma) \cap E\big) - \frac{\beta \delta}{2}\, \mu(B) > \frac{\delta}{2} - \frac{\beta\delta}{2} \ge 0.\]
But note that 
\begin{align*}
\gamma(D(\gamma) \cap E) &\subset \bigcup_{n \in \Z} \Big( T^{d_0} \times T^{a_1 + d_0 - a_0}\Big) \circ \left(T\times T\right)^n \circ \Big( T^{-a_0} \times T^{- a_1}\Big) (D(\gamma) \cap E) \\
& = \bigcup_{n \in \Z} \Big(\big(T^{d_0 - a_0} \times T^{d_0 - a_0} \big) \circ \left(T\times T\right)^n \Big) (D(\gamma) \cap E) \\
& = \bigcup_{n \in \Z} \left(T\times T\right)^n (D(\gamma) \cap E).
\end{align*}
This set must have positive intersection with $F$, whence for some $n \in \Z$, we have $\mu\left((T\times T)^n E \cap  F\right) > 0$. A similar proof holds when $d_0 - a_0 \ge d_1 - a_1$. 
\end{proof}

The next lemma provides techniques to prove that more general transformations are not ergodic. 

\begin{lemma}\label{L:cedescendants} 
	Let $S:=T^{\alpha_0} \times  \cdots \times T^{\alpha_{k-1}}$ be a product of nonzero-integer powers of rank-one transformations in the product space $X:= Y \times \cdots \times Y$. If $T$ is conservative ergodic, for every $\epsilon > 0$, $i \in \N$, $I$ the base of $C_i$, and $(b_0,...,b_{k-1}) \in \{0,...,h_i - 1\}^k$, there exists a natural number $j>i$ such that for at least $(1 - \ve)|D(I,j)|^k$ tuples of descendants $(a_0, ..., a_{k-1}) \in D(I, j)^k$ we have $a_\ell = d_\ell +b_\ell + \alpha_\ell n$ for $\ell = 0,...,k-1$ for some tuple $(d_0,...,d_{k-1}) \in D(I,j)^k$ and $n\in \Z\setminus \{0\}$. 
\end{lemma}

\begin{proof}
Fix any $\ve>0$. First, if $S$ is conservative ergodic, then we must be able to find a natural number $m$ such that $A$ is covered by $\bigcup_{n= -m, \, n \ne 0}^m S^n B$ up to a measure $\frac{\ve}{2}\, \mu(A)$. 

Recall that for any $j \ge i$, all of the rectangles in $\left(C_j\right)^k$ are pairwise disjoint, and $A$ and $B$ are the disjoint unions of such rectangles. In addition, for any $j$-subrectangle $C$ of $B$ of the form $C=T^{c_0 + b_0}J \times \cdots \times T^{c_{k-1} + b_{k-1}} J$ (where $J$ is the base of $C_{j}$), if $c_\ell \in D (I,j)$, $c_\ell \ge m|\alpha_\ell| + |b_\ell|$ and $c_\ell < h_{j} - m|\alpha_\ell| - |b_\ell|$ for all $\ell$, then $S^n(C),\, |n| \le m$ is also a rectangle in $\left(C_j\right)^k$. But the proportion of such rectangles $C$ grows arbitrarily high in $j$. Thus, we can choose $j$ large enough such that $\bigcup_{n= -m, \, n \ne 0} S^n B$ is composed up to a measure $\frac{\ve}{2}\, \mu(A)$ by a union of rectangles that are elements of $\left(C_j\right)^k$ and fall entirely inside of $\bigcup_{n= -m, \, n \ne 0} S^n B$. Specifically, we use rectangles of the form 
\begin{equation} \label{E:cedescendants1}
T^{d_0+ b_0+\alpha_0 n} J \times T^{d_1 + b_1 + \alpha_1 n} J \times \cdots \times T^{d_{k-1} + b_{k-1} + \alpha_{k-1}n} J
\end{equation}
for $|n|\le m$, $n \ne 0$, where $J$ is the base of $C_{j}$ and $(d_0,...,d_{k-1}) \in  D(I,j)^k$. Then by supposition these rectangles must cover $A$ up to a measure $\ve \mu(A)$. But because the rectangles of $\left(C_j\right)^k$ are pairwise disjoint, a $j$-rectangle in $\bigcup_{n= -m, \, n \ne 0} S^n B$ intersects with $A$ only if it equals a subrectangle of $A$. We proceed to denote the covered subrectangles of $A$ by $k$-tuples $(a_0,...,a_{k-1}) \in D(I,j)^k$, and note that such subrectangles must equal rectangles of the form given in (\ref{E:cedescendants1}). 

This implies that, for at least $(1-\ve)|D(I,j)|^k$ of the $k$-tuples $(a_0,...,a_{k-1}) \in D(I,j)^k$, we must have the relation
\begin{align}
T^{a_0} J \times \cdots \times T^{a_{k-1}} J = T^{d_0+ b_0 + \alpha_0 n} J  \times \cdots \times T^{d_{k-1} + b_{k-1} + \alpha_{k-1}n} J \label{E:situation}
\end{align}
for some nonzero $n, \, |n|\le m$, and some $(d_0,...,d_{k-1}) \in D(I,j)^k$ such that $d_\ell + b_\ell + \alpha_\ell n,\, |n| \le m$ is defined as a level in $C_j$. This can only happen if we have $a_\ell = d_\ell + b_\ell +\alpha_\ell n$ for all $\ell = 0,...,k-1$. 

\end{proof}

These lemmas now yield the following propositions, which provide necessary conditions and stronger sufficient conditions for the ergodicity of a transformation $T$ which is a product of powers of a rank-one transformation $T$.

\begin{prop} \label{P:kproduct} For a rank-one transformation $T$,  $T\times T$ is conservative ergodic if for every  $\ve>0$, $i\in \N$, $I$ the base of $C_i$, and $0 \le b < h_i-1$, there is a natural number $j>i$ such that at least $(1-\ve)|D(I,j)|^k$ $k$-tuples of descendants of the base $I$ of column $C_i$ of the form $(a,a') \in D(I,j)^k$ have uniquely corresponding pairs $(d,d') \in D(I,j)^k$ such that $a-d = a'-d' - b$.
\end{prop}
\begin{proof}

Suppose that $T$ meets the stated condition, and fix an $i \in \N$. Let $A = I\times I$ and $B = I\times T^b I$ for $I$ the base of some column $C_i$ of $T$, and an integer $b \in \{0,...,h_i-1\}$. Fix some positive $\ve$---then by supposition, there exists some $j> i$ such that some fraction $1-\ve$ of the $|D(I,j)|^2$ of the constituent subrectangles of $A$ of the form $T^a J \times T^{a'} J,\, (a,a') \in D(I,j)^2$ can be associated with unique complementary descendant tuples $(d,d')$ with $a-d' = a' - d' - b$. Letting $n = d - a$, this implies that 
\[\left(T\times T\right)^n \left(T^{a} J \times T^{a'} J\right) = T^{d} J \times T^{d' + b} J \subset B.\]
Define $\tau_{(a,a')}$ to be any measure-preserving bijection taking points from $T^{a} J \times \cdots \times T^{a'} J$ to $T^{d} J \times T^{d' + b} J$. Let $F(I,j) \subset D(I,j)^2$ be the set of all pairs $(a,a')$ of descendants with such a unique complementary tuple $(d,d')$. Then we can define 
\[\tau = \bigsqcup_{(a,a') \in F(I,j)} \tau_{(a,a')},\]
in which case $\tau \in [[T\times T]]$, $D(\tau) \subset A$, and $R(\tau) \subset B$. Also, $\mu(D(\tau)) \ge (1-\ve)\, \mu(A)$ and $\tau$ is measure-preserving. Because $\ve$ can be taken arbitrarily small, lemma \ref{L:semiringce} (with any choice of $\beta, \delta \in (0,1)$) implies that $T\times T$ is ergodic.  
\end{proof}

Lemma \ref{L:cedescendants} also suggests a method of establishing non-ergodicity of products of powers of rank-one transformations. 

\begin{prop} \label{P:arbproduct} For $T$ a rank-one transformation and nonzero integers $\alpha_0,...,\alpha_{k-1}$, $S:=T^{\alpha_0} \times \cdots \times T^{\alpha_{k-1}}$ is conservative ergodic only if for every $\ve>0$, $i\in \N$, $I$ the base of $C_i$, and $k$-tuple $(b_0,...,b_{k-1}) \in \{0,...,h_{i}-1\}^k$, there is a natural number $j>i$ such that for at least $(1-\ve)|D(I,j)|^k$ $k$-tuples of descendants of the form $(a_0,...,a_{k-1}) \in D(I,j)^k$, we have corresponding $k$-tuples $(d_0,...,d_{k-1}) \in D(I,j)^k$ such that $\frac{a_0 - d_0 - b_0}{\alpha_0}= \frac{a_\ell - d_\ell - b_\ell}{\alpha_\ell} \in \Z\setminus \{0\}$ for each $\ell = 0,...,k-1$. $T$ is ergodic if this condition holds, and to every tuple $(a_0,...,a_{k-1})$ with a complementary $k$-tuple meeting the stated conditions, we can associate a unique such complementary tuple.  
\end{prop}
\begin{proof}
By lemma \ref{L:cedescendants}, if $S$ is ergodic, for every $i \in \N$ and $\ve> 0$ we can find some $j > i$ such that $(1-\ve)|D(I,j)|^k$ $k$-tuples of descendants have complementary tuples $(d_0,...,d_{k-1}) \in D(I,j)^k$ satisfying $a_\ell = d_\ell + b_\ell + \alpha_\ell n$ for some $n \in \Z\setminus \{0\}$. This implies that $\frac{a_0 - d_0 - b_0}{\alpha_0} = \frac{\alpha_\ell - d_\ell - b_\ell}{\alpha_\ell } = n  \in \Z\setminus \{0\}$ for each $\ell$. The proof of the sufficiency of the second stated condition is similar to the proof of Proposition \ref{P:kproduct}, and is omitted.
\end{proof}



\section{Combinatorics} \label{S:comb}

We have shown that a rank-one cutting and stacking transformation can be characterized by its descendant sets, which encapsulate information about its cutting and stacking parameters. Descendant sets are just sum sets of height sets, which each correspond to cuts and spacers added to one particular column. The following lemma is used to construct the height sets $H_k$ for a rank-one transformation $T$ such that $T\times T$ is ergodic but $T\times T^{-1}$ is not. 

\begin{lemma}\label{L:comb}
	Let $M, \Gamma, \gamma \in \N$. Then there are sets of nonnegative integers $H(U), H(L)$, where $H(U) = \{\{V_1, W_1\}, \ldots, \{V_{\Gamma}, W_{\Gamma}\}\}$ and $H(L) = \{\{v_1,w_1\}, \ldots, \{v_{\gamma}, w_{\gamma}\} \}$, and letting \[H = \{ V_i, W_j, v_k, w_{\ell} \mid 1\le i, j\le \Gamma, 1\le k, \ell, \le \gamma \},\] $H$ satisfies the following properties:
	\begin{enumerate}
		\item For every $\{V,W\}\in H(U)$ and $\{v,w\} \in H(L)$ we have $V+W=v+w-1$
		\item If $x_1, x_2, x_3, x_4$ are in $H$ and $|x_1+x_2-x_3-x_3| < M$, then precisely one of the following holds:
		\begin{itemize}
			\item $\{x_1,x_2\} = \{x_3,x_4\}$
			\item $\{x_1,x_2\} \neq \{x_3,x_4\}$ but $x_1 + x_2 = x_3 + x_4$, in which case $\{x_1,x_2\}$ and $\{x_3,x_4\}$ are both in either $H(U)$ or $H(L)$,
			\item $x_1 + x_2 = x_3 + x_4 - 1$, in which case $\{x_1, x_2\}\in H(U)$ and $\{x_3, x_4\} \in H(L)$ , or
			\item $x_1 + x_2 = x_3 + x_4 + 1$, in which case $\{x_1, x_2\}\in H(L)$ and $\{x_3, x_4\} \in H(U)$ .
		\end{itemize}
	\end{enumerate}
\end{lemma}

\begin{proof}
	We proceed by finding a set $H$ such that $V_r+W_r=v_s+w_s$ for all $r\in \{1,\dots,\Gamma\}$ and $s\in \{1,\dots,\gamma\}$, and such that $|x_1+x_2-x_3-x_4|<1$ with distinct summands implies $\left\lbrace x_1,x_2, x_3,x_4\right\rbrace$ is one of $\{v_s,w_s,v_{s'},w_{s'}\}$, $\{v_s,w_s,V_r,W_r\}$, or $\{V_r,W_r,V_{r'},W_{r'}\}$. For this construction of $H$ when $M=1$, choose $n\gg 2^{2(\Gamma+\gamma)}$ and even, and let 
	\[H:=\Big\lbrace 2, \dots, 2^{\Gamma+\gamma}, n-2^{\Gamma+\gamma}, \dots , n-2\Big\rbrace,\]
	 where $H(U) = \{\{2, n-2\}, \dots, \{2^{\Gamma}, n-2^{\Gamma}\}\}$ and $H(L) = \{ \{2^{\Gamma+1}, n-2^{\Gamma+1}\}, \dots, \{2^{\Gamma+\gamma}, n-2^{\Gamma+\gamma}\}\}$. For $r\in \{1, \dots,\Gamma\}$ let $V_r=2^r$, $W_r=n-2^r$ and for $s\in \{1,...,\gamma\}$ let $v_s=2^{\Gamma+s}$ and $w_s=n-2^{\Gamma+s}$.

	Now, partition $H$ into sets $R_1=\{2, \dots, 2^{\Gamma+\gamma}\}$ and $R_2=\{n-2^{\Gamma+\gamma}, \dots, n-2\}$. Note that, given four elements $x_1,x_2,x_3,x_4\in H$ with $|x_1+x_2-x_3-x_4|<M=1$, we have $x_1+x_2=x_3+x_4$. Suppose that $x_1,x_2\in R_1$: that is, $x_1=2^{z_1}$ and $x_2=2^{z_2}$, for integers $0\le z_1,z_2\le \Gamma+\gamma$. Then $x_1+x_2 \le 2^{\Gamma+\gamma+1} \ll n-2^{\Gamma+\gamma}$, so $x_3$ and $x_4$ are also both in $R_1$. By unique binary expansion, either $z_1=z_3$ and $z_2=z_4$ or $z_1=z_4$ and $z_2=z_3$. Then $\{x_1,x_2\}=\{x_3,x_4\}$, so we obtain the first subcase above. Suppose that $x_1\in R_1$, $x_2\in R_2$. Then $x_3$ and $x_4$ are not both in $R_1$ and the size of $n$ dictates that precisely one of $\{x_3,x_4\}$ is in $R_1$. Without loss of generality write $x_1=2^{z_1},\, x_2=n-2^{z_2},\, x_3=2^{z_3},\text{ and } x_4=n-2^{z_4}$ where $z_1,z_2,z_3,z_4\in \{1,...,\Gamma+\gamma\}$. Then we obtain $2^{z_1}+2^{z_4}=2^{z_2}+2^{z_3}$, implying that either $z_1=z_2$ and $z_3=z_4$ or $z_1=z_3$ and $z_2=z_4$. In the former case $x_1,x_2$ are a pair $\{v,w\}$ or $\{V,W\}$ and $x_3,x_4$ also form such a pair; in the latter case because then $x_1=x_3$ and $x_2=x_4$, so $\{x_1,x_2\}=\{x_3,x_4\}$. Symmetry addresses the case where $x_1\in R_2$, $x_2\in R_1$. Finally, if $x_1,x_2\in R_2$ then both $x_3$ and $x_4$ are in $R_2$; setting $x_1=n-2^{z_1},\, x_2=n-2^{z_2},\, x_3=n-2^{z_3},$ and $x_4=n-2^{z_4}$, we see that $2^{z_1}+2^{z_2}=2^{z_3}+2^{z_4}$, which again implies the first subcase. Hence, $H$ conforms to its stated condition.

	Fix a $M\in \N$ with $M\ge 2$. Multiply every element in $H$ by $M$, and then subtract $1$ from all of the elements obtained from multiplying $M$ with a $V_r$. Call $V_r'=M\cdot V_r-1$, $W_r'=M\cdot W_r$, and so on. Call the set containing these new pairs $H'$. Suppose that $y_1,y_2,y_3,y_4$ are distinct elements in $H'$ with $|y_1+y_2-y_3-y_4|<M$. Let $x_1,x_2,x_3,x_4$ be their corresponding elements in $H$. By adding $1$ to all $y$ terms of the form $V_r'$, we obtain that $|Mx_1+Mx_2-Mx_3-Mx_4|<M+2$, whence $|x_1+x_2-x_3-x_4|<1+\frac{2}{M}\le 2$. So $|x_1+x_2-x_3-x_4|=0$ or $1$. But recall that $n$ was chosen to be even, so $|x_1+x_2-x_3-x_4|=0$. Thus, the pairs $\{x_1, x_2\}$ and $\{x_3, x_4\}$ are either both in $H(U)$ or $H(L)$ or are split evenly between them, which implies the same for $\{y_1,y_2\}$ and $\{y_3,y_4\}$ in $H(U)'$ and $H(L)'$. Hence, $H'$ is our desired set for any given $M$, when we let $H(U)'$ be the set of pairs $\{V_r',W_r'\}$ and $H(L)'$ be the set of pairs $\{v_s',w_s'\}$.
\end{proof}

\begin{remark}\label{R:construction}
Using Lemma \ref{L:comb}, we can construct the height sets $H_k$ of our transformation inductively. Specifically, let $M_k,\Gamma_k,\gamma_k\in \N$ be the inputs for set $H_k$, as implemented in Lemma \ref{L:comb}. Choose \[M_k \gg 2 \max D(I, k) = 2 \max (H_0 \oplus H_1 \oplus \cdots \oplus H_{k-1})\] 
This clearly ensures that the difference between any two elements in $H_k$ is greater than $h_{k-1}$. As of yet, let $\{\Gamma_k\}$ and $\{\gamma_k\}$ remain unspecified; we choose them towards the end of our construction.
\end{remark}

For reasons that will soon become clear, we need to categorize pairs in $H_k^2$ by their additive properties. Certain pairs $(a,a')$ drawn from $H_k^2$ will have complementary pairs $(d,d') \in H_k^2$ that satisfy $a-d = a' - d' -1$, and others will have complements $(d,d')$ satisfying $a + d = a' + d' + 1$. Our goal is to maximize the proportion of the former in order to make $T\times T$ ergodic, but minimize the proportion of the latter in order to keep $T\times T^{-1}$ from being ergodic. 

\begin{definition}
	Let $H$ be as in Lemma \ref{L:comb}. 
	A pair $\{x, y\} \in H\times H$ is called {\bf mixed} if $x=V_i$ or $W_j \in H(U)$, and $y = v_k$ or $w_{\ell} \in H(L)$, or vice-versa. 
	A mixed pair is called {\bf positive} if it is of the form $(w_j, W_i), (w_j, V_i), (v_j, V_i)$ or $(v_j, W_i)$. A pair is called {\bf negative} if it is of the form $(V_i, v_j), (V_i, w_j), (W_i, v_j)$ or $(W_i, w_j)$. A negative mixed pair will be said to \textbf{correspond} to a positive mixed pair $(d,d')$ if $a - d = a' - d' - 1$ (for instance, $(V_i,v_j)$ corresponds to $(w_j,W_i)$). Note that this correspondence is one-to-one for any negative mixed pair. 

	A pair $\{x, y\} \in H$ is called {\bf pure} if $\{x, y\} \in H(L)$ or $\{x, y\} \in H(U)$. Notice that the pure pairs are unordered, whereas the mixed pairs are ordered (and are positive or negative depending upon the order of the elements).
\end{definition}

The use of the words ``positive'' and ``negative'' is meant to be evocative. Let $a, a' \in D(I, j)$, and let $b$ be fixed. As in Theorem \ref{T:invcons}, we write $a = \sum_{k=i}^{j-1} a_k$ where $a_k \in H_k$. As established in the preceding lemmas, we are interested in necessary and sufficient conditions for, for instance, the existence of $d, d' \in D(I, j)$ such that $a - d = a' - d' - b$. If there are $b$ indices $k$ such that $\{a_k, a_k'\}$ is negative mixed, then we can satisfy this condition; choose $d_k, d_k'$ to be the corresponding positive mixed pair to get $a_k - d_k = a_k' - d_k' - 1$ for those $b$ indices, and for the remainder set $d_k = a_k$ and $d_k' = a_k'$. Crucially, given any such pair $(a,a')$ with $b$ indices having $(a_k,a_k')$ negative mixed, we can associate a unique pair $(d,d')$ satisfying $a-d = a'-d' - b$. There is a similar idea for dealing with the condition relating to $U$, that is, $a + a' = d + d' = b$.

\begin{lemma}\label{L:mixed}
	Let $n$ be fixed and let $M_k$ be the increasing sequence discussed in remark 1 with $M_0>1$. Let $I$ be the base level of $C_i$, where $i<n$, and suppose that $a + a' = d + d' + 1$, with $a, a', d, d' \in D(I, n)$. Write $a = \sum_{k = i}^{n-1} a_k$ with $a_k \in H_k$, and similarly for $d, a', d'$. Then there is a $k$ in $\{i, \dots, n-1\}$ such that $\{a_k, a_k'\} \in H_k(U)$ and $\{d_k, d_k'\} \in H_k(L)$, or vice-versa.
\end{lemma}
\begin{proof}
	We clearly cannot have $a_k + a_k' = d_k + d_k'$ for each $k$, so choose the largest $k$ such that equality does not hold. Recall that $M_k$ is the constant used to construct $H_k$ in Lemma \ref{L:comb}, and was chosen to be $\gg 2 \max D(I, k)$ in Remark \ref{R:construction}. The first case is $|a_k + a_k' - d_k + d_k'| < M_k$. So, we have that $\{a_k, a_k'\}$ and $\{d_k, d_k'\}$ must be pairs in $H_k(U)$ and $H_k(L)$.  So we have $\{a_k, a_k'\} \in H_k(U)$ and $\{d_k d_k'\} \in H_k(L)$ or $\{a_k, a_k'\} \in H_k(L)$ and $\{d_k, d_k'\} \in H_k(U)$.

	The second case is when $|a_k + a_k' - d_k - d_k'| \ge M_k \gg 2 \max D(I, k)$. We have
	\begin{align*}
	|a + a' - d - d'| &= \left| \sum_{j = i}^{n-1} a_j + \sum_{j = i}^{n-1} d_j - \sum_{j = i}^{n-1} a_j' - \sum_{j = i}^{n-1} d_j' \right|\\
	&= \left| \sum_{j=i}^{k} (a_j + d_j - a_j' - d_j' ) \right|\\
	& \geq |a_k + d_k - a_k' - d_k'| -  \sum_{j = i}^{k-1} \left| (a_j + d_j - a_j' - d_j') \right|\\
	& \geq M_k - 2 \sum_{j=1}^{k-1} \max H_j \\
	& = M_k - 2 \max D(I, k) \gg 1,
	\end{align*}
	which contradicts the initial assumption, concluding the lemma.
\end{proof}


\section{For each rank-one  $T$ and $n \in \Z\setminus \{0\}$, $T^n\times T^{-n}$ is conservative}\label{S:ttinvcons}

We note that there exist rank-one transformations $T$ such that $T\times T$ is not conservative \cite{AdFrSi97}, as well as infinite measure-preserving transformations where $T\times T^{-1}$ is not conservative (Corollary~\ref{C:tt-1notcons}). The proof of  Lemma~\ref{L:krectangles} below follows from the proof of Proposition 8.1 in \cite{DJJSS10}.

\begin{lemma}\label{L:krectangles}
	Let $T$ be any infinite measure-preserving transformation on $X$ and $\mathcal{D}$ be a sufficient semiring in $X$. Suppose that $T$ satisfies the conservativity conditon on $\mathcal{D}$ that for every $A \in \mathcal{D}$ we have $A \subset \bigcup_{n \in \Z\setminus\{0\}} T^n A \mod{\mu}$. Then $T$ is conservative.
\end{lemma}

This lemma has a very desirable equivalence between a property on our semiring and a property on all of our measure space. It allows us to use our descendant's notation to its fullest potential. Using Lemma \ref{L:krectangles}, we have the following equivalent condition to conservativity of products of rank-one transformations: 

\begin{prop} \label{L:conservprods2}
	Let $T$ be a rank-one transformation on a measure space $Y$ and $A = I \times \cdots \times I$ ($k$ times), where $I$ is the base of column $C_i$. Furthermore, let $(\alpha_0,...,\alpha_{k-1})$ be a $k$-tuple of nonzero integers. Set $\mathcal{D}$ to be the semiring of rectangles in $X:=Y\times \cdots \times Y$ ($k$ times) which have levels of columns of $T$ as sides. Then the product transformation $S:=T^{\alpha_0}\times \cdots \times T^{\alpha_{k-1}}$ on $X$ is conservative if and only if for every $\ve > 0$ there is $j$ such that at for at least $(1 - \ve)|D(I,j)|^k$ of the $k$-tuples $(a_0,...,a_{k-1}) \in D(I,j)^k$, there exist complementary $k$-tuples $(d_0,...,a_{k-1}) \in D(I,j)^k$ satisfying $\frac{a_0-d_0}{\alpha_0} = \frac{a_\ell- d_\ell}{\alpha_\ell} \in \Z\setminus \{0\}$ for $\ell = 1,...,k-1$. 
	\end{prop}

\begin{proof}
Fix $\ve>0$. First, if $S$ is conservative, we can find some $m$ such that $A$ is covered by $\bigcup_{n=-m,n\ne 0}^m S^n A$ except for some measure $\frac{\ve}{2} \, \mu(A)$. Then we may choose $j$ large enough such that, up to measure $\frac{\ve}{2} \, \mu(A)$, all of the intersections $S^n A \cap A$ for $n,\, |n| \le m$ are composed of unions of rectangles in $\left(C_j\right)^k$ for $\ell = 0,...,k-1$ (see the proof of Lemma \ref{L:cedescendants}). These subrectangles are descendant rectangles of $A$ (that is, their sides are descendants of the original sides of $A$) whose sides have heights indexed by elements in $D(I,j)^k$. Thus, out of a total of $|D(I,j)|^k$ subrectangles at the $j^\text{th}$ stage, at least $(1-\ve)|D(I,j)|^k$ are contained in $S^n A \cap A$ for some $n \ne 0$. This implies that an equal number of $k$-tuples $(a_0,...,a_{k-1})\in D(I,j)^k$ will satisfy $T^{a_0} J \times \cdots \times T^{a_{k-1}} J \subset S^n A$ for some $n\ne  0$ with $|n| \le m$. For these rectangles, this can only happen if $a_\ell = d_\ell + \alpha_\ell n$ for some $n \in \Z$ and all $\ell = 0,...,k-1$, for some $k$-tuple $(d_0,...,d_{k-1}) \in D(I,j)^k$. 

Now suppose that the conditions of the lemma hold for $S$. Then we may choose $m$ so large that, up to a measure $\ve\mu(A)$, all of the $(1-\ve)|D(I,j)|^k$ subrectangles of $A$ are contained in $\bigcup_{n=-m,\, n\ne 0}^m S^n A$. Specifically, note that if $\frac{a_0-d_0}{\alpha_0} = \frac{a_\ell- d_\ell}{\alpha_\ell} \in \Z\setminus \{0\}$ for all $\ell$, then we should have 
\[
T^{a_0} J\times \cdots \times T^{a_{k-1}} J = S^n\big(T^{d_0} J \times \cdots \times T^{a_{k-1}} J \big)\subset S^n A,
\]
where $n = -\frac{a_0-d_0}{\alpha_0}$. Then $A$ is covered up to measure $\ve \mu(A)$ by $\bigcup_{n=-m,\, n\ne 0}^m S^n A$. Our choice of $\ve$ was arbitrary, so we must have $A \subset \bigcup_{n=-m,\, n\ne 0}^m S^n A \mod{\mu}$ for some $m \in \N$, and all sets $A$.

We know show that all sets in $\mathcal{D}$ have the same property: cut any rectangle in $\mathcal{D}$ into a disjoint union of rectangles in $\left(C_i\right)^k$ for some $i \in \N$, and $I$ the base of $C_i$. Then we may write each such constituent subrectangle as $T^{b_0} I \times \cdots \times T^{b_{k-1}} I$, where $b_\ell \in \{0,...,h_{i}-1\}$ for $\ell = 0,...,k-1$. 
Note that
\begin{align*}
T^{b_0} I \times \cdots \times T^{b_{k-1}} I  &= \left(T^{b_0} \times \cdots \times T^{b_{k-1}}\right) I \times \cdots \times I 
\\
&\subset \left(T^{b_0} \times \cdots \times T^{b_{k-1}}\right)\bigcup_{n \in \Z\setminus \{0\}} S^n\left(I \times \cdots \times I\right)
\\
& = \bigcup_{n \in \Z \setminus \{0\}} S^n \left( T^{b_0} I \times \cdots \times T^{b_{k-1}} I \right).
\end{align*}
This result holds for each of the constituent subrectangles of any rectangle in $\mathcal{D}$, so lemma \ref{L:krectangles} yields that $S$ is conservative. 
\end{proof}

For the proof of the following theorem, it will be helpful to have notation by which we can break up elements of $D(I,j)$ into their additive components. To this end, for any element $a \in D(I,j)$, we can and will write $a = \sum_{k=i}^{j-1} a_k$ where $a_k \in H_k$, by the decomposition $D(I, j) = H_i \oplus \cdots \oplus H_{j-1}$. We proceed similarly for $a',d$, and $d'$.

\begin{theorem}\label{T:invcons}
	Let $T$ be a rank-one transformation, and $n$ a nonzero integer. Then $T^n \times T^{-n}$ is conservative.
\end{theorem}
\begin{proof}
	Let $A = I \times I$, where $I$ is the base of any column $C_i$. It suffices to show (by Proposition \ref{L:conservprods2}) that for every $\epsilon > 0$ there is $j$ such that at least $(1-\ve)|D(I,j)|^2$ of the pairs $(a,a') \in D(I,j)^2$ have complementary pairs $(d,d') \in D(I,j)^2$ with $a - d = d' - a' \in n \Z \setminus \{0\}$. 

	For any $k \in \N$, let $R'(I,k) \subset D(I,k)^2$ denote the largest subset of $D(I,k)^2$ containing pairs $(a,a')$ which have $n\, \big|\, a-a'$. A pigeonhole argument shows that $|R'(I,k)| \ge \frac{1}{n^2} \, |D(I,k)|^2$ (take all pairs of elements drawn from the largest intersection of $D(I,k)$ with a congruence class modulo $n$). Of all of the pairs in $R'(I,k)$, at most $|D(I,k)|$ have $a = a'$. Hence, at least $\frac{1}{n^2} \, |D(I,k)|^2 - |D(I,k)|$ of the pairs in $D(I,k)^2$ have $n\, \big|\, a-a'$ and $a \ne a'$. Call the set of all such pairs $R(I,k)$, and note that $\frac{|R(I,k)|}{|D(I,k)|^2} \ge \frac{1}{n^2} - \frac{1}{|D(I,k)|}$. Recalling that $D(I,k) \ge 2^{k- i}$ and is monotonically increasing in $k$ whenever $k \ge i$, we let $k'$ denote the smallest integer $k \ge i$ such that $|D(I,k)| > 2n^2$. 
	
	For any $j > k'$, consider any pair $(a,a') \in D(I,j)^2$ which has $(a_k,a'_k) \in R(I,k)$ for some $k$ with $k' \le k < j$. Denote this particular value of $k$ by $k^*$. We construct $(d,d') \in D(I,j)^2$ as follows: for $k \ne k^*$, set $d_k = a_k$ and $d_k' = a_k'$. Then set $d_{k^*} = a_{k^*}'$ and $d_{k^*}' = a_{k^*}$. Then we clearly have $a + a' = d+ d'$, and thus $a - d = d' - a' = a_{k^*} - a'_{k^*} \in n \Z \setminus \{0\}$. Furthermore, the proportion of such pairs inside of $D(I,j)^2$ is lower-bounded by 
	\begin{align*}
	1 - \prod_{k = k'}^{j-1} \left( 1- \frac{|R(I,k)|}{|D(I,j)|^2}\right) &
\ge 1- \prod_{k=k'}^{j-1} \left(\frac{1}{n^2} - \frac{1}{|D(I,k)|} \right)
	\\
	 &\ge  1 - \left( 1- \frac{1}{2n^2}\right)^{j-k'}.
	\end{align*}
This quantity goes to $1$ as $j$ grows large, so we may conclude that $T^n \times T^{-n}$ is conservative. 
\end{proof}


\section{$T\times T$  ergodic but $T\times T^{-1}$  not ergodic}\label{S:Tinv} 

In this section, we use the combinatorial results of Section \ref{S:comb} to construct a class of rank-one transformations $T$  such that $T\times T$ is ergodic but $T\times T^{-1}$  is not ergodic.
To obtain ergodicity of the Cartesian square we just need  $\gamma_k = \Gamma_k$ for all $k$
with  arbitrary $\Gamma_k$.

\begin{theorem}\label{T:txterg}
	Let $T$ be defined using the height sets given in Lemma \ref{L:comb}, and by setting $\gamma_k = \Gamma_k > 0$ for every $k \ge 0$. Then $T^{(2)}$ is ergodic.
\end{theorem}
\begin{proof}
	Suppose that $T$ is as specified. We will apply Proposition \ref{P:kproduct}. To do so, for any $i \in \N$, $\ve > 0$, $b \in \{0,..,h_i -1 \}$, we must show that there exists a natural number $j > i$ such that at least $(1-\ve)\, |D(I,j)|^2$ pairs $(a,a')$ of descendants of the base $I$ of $C_i$ can be associated to unique complementary pairs $(d,d') \in D(I,j)^2$ satisfying $a - d = a'-d' - b$. 
	
	Recall that the descendants of $I$ in $C_j$ can be given by $D(I,j) = H_i \oplus H_{i+1} \oplus \dots \oplus H_{j-1}$. So we can decompose any element $c \in D(I,j)$ into its sum components as $c = \sum_{k=i}^{j-1} c_k$, where $c_k \in H_k$. We will employ this notation for pairs $(a,a')$ and their corresponding pairs $(d,d')$. Given any $k \ge i$, also recall that $H_k$ contains $2\gamma_k + 2 \Gamma_k = 4\gamma_k$ elements, so $H_k^2$ contains $16\gamma_k^2$ elements. On the other hand, the number of negative mixed pairs in $H_k$ is given by $4 \gamma_k^2$. So the proportion of pairs in $H_k^2$ which are negative mixed is $1/4$. Thus, taking any $j\gg i+b$, the proportion of pairs $(a,a')$ in $D(I,j)^2$ which have $(a_k,a'_k)$ negative mixed for $k = i,...,i+b-1$ is $\frac{1}{4^b}$. For any such pair, we can take $(d_k,d_k')$ to be the positive mixed pair corresponding to $(a_k,a'_k)$ for $k = i,...,i+b-1$, and $d_k = a_k,\, d_k' = a_k'$ elsewhere. Then we should clearly have $a - d = a'- d' - b$, and to every such $(a,a')$ we can associate a unique pair $(d,d')$ (as $(d_k,d_k')$ is the unique positive mixed pair corresponding to $(a_k,a_k')$ whenever the pairs are not chosen to be exactly equal). 
	
	We have now deduced that at most $1-\frac{1}{4^b}$ of the pairs in $D(I,j)^2$ do not meet the precondition for ergodicity given in Proposition \ref{P:kproduct} whenever $j \gg i+b$. Taking $j$ sufficiently high and then considering the proportion of pairs $(a,a') \in D(I,j)$ which have $(a_k,a_k')$ as negative mixed pairs in $H_{i+b},...,H_{i+2b-1}$, which is independent of the previous case, we can reduce this proportion to $\left(1- \frac{1}{4^b}\right)^2$. Continuing in this manner, the proportion of unsatisfactory pairs may be made arbitrarily small, from which we can deduce that $T\times T$ is ergodic. 
\end{proof}

\begin{theorem}\label{T:ieinotcons}
	Let $T$ be a rank-one transformation constructed using a sequence $0 < \{\gamma_\ell\}$ that satisfies 
	\[
		0 < \prod_{\ell \in \N} \left(1 - \frac{1}{4\gamma_\ell}\right) 
	\]
	and $ \gamma_k=\Gamma_k $ for all $k \ge 0$. Then $T\times T$ is ergodic but $U = T \times T^{-1}$ is not ergodic.
\end{theorem}
\begin{proof} Ergodicity of $T\times T$ follows from Theorem~\ref{T:txterg}. 
	We will proceed by contradiction by supposing that $U$ is conservative ergodic. Letting $I$ be the base of an arbitrary column $C_i$, let $A = I\times I$ and $B = I \times T(I)$ (that is, choose $b_0 = 0$ and $b_1 = 1$). Then by Proposition \ref{P:arbproduct}, for every $\epsilon > 0$ there exists $j$ such that for at least $(1 - \epsilon)|D(I,j)|^2$ pairs of descendants $(a, a') \in D(I, j)^2$ we have $a+ a' = d + d' + 1$ for some $n$. By Lemma \ref{L:mixed}, this occurs for a pair $(a,a')$ only if there exists some $k\in \{i,\ldots,n-1\}$ such that $(a_k,a_k')$ is a pure pair. But there are only $2(2\gamma_k)$ possible pure pairs in $H_k$ out of $16\gamma_k^2$ total pairs. Let $P \subset D(I,j)^2$ denote the set of pairs $(a,a')\in D(I,j)^2$ such that $(a_k,a_k')$ is never pure for $i \le k < j$. Then the proportion of pairs in $D(I,j)^2$ with at least one additive component pair $(a_k,a_k'),\, i \le k < j$ pure is $\frac{|P|}{|D(I,j)|^2}$, and 
	\begin{align*}
		\frac{|P|}{|D(I,j)|^2} &= 1 - \frac{|P^c|}{|D(I,j)|^2}\\
		&= 1-\prod_{\ell = i}^{j-1} \left(1 - \frac{1}{4\gamma_\ell}\right)
	\end{align*}
	Since this quantity is strictly less than $1$, the proportion of pairs $(a,a') \in D(I,j)$ with a complementary pair $(d,d')$ satisfying $a+ a' = d+ d' + 1$ must be bounded above by a number less than $1$. For any choice $0< \epsilon <\prod_\ell (1 - 1/4\gamma_\ell)  $, this contradicts ergodicity of $U$.
\end{proof}

Regarding ergodicity of higher products, we note that $T\times T\times T$ ergodic is equivalent to
 the statement that for any $b_0, b_1, b_2$ and $I$ the base of some column, the proportion of triples $(a_0, a_1, a_2) \in D(I, j)^3$ having uniquely associated corresponding descendant triples $(d_0, d_1, d_2) \in D(I, j)^3$ with
\[a_0 - d_0 - b_0 = a_1 - d_1 - b_1 = a_2 - d_2 - b_2\]
goes to 1 as $j\to\infty$. We can write this in a slightly nicer form, letting $b_0 = b$ and $b_1 = b_2 = 0$, as
\begin{align*}
a_0 + d_1 &= a_1 + d_0 + b\\
a_0 + d_2 &= a_2 + d_0 + b.
\end{align*}
It remains open as to whether this condition can correspond to $T \times T^{-1} $ not being ergodic.

\section{A Markov shift with $T \times T^{-1}$ not conservative}

In this section we construct a conservative ergodic Markov shift $T$ such that $T\times T^{-1}$ is not conservative. This is based on the examples of Kakutani and Parry \cite{KaPa63}. For further background and terms not defined below regarding Markov shifts, the reader is referred to \cite{Aa97}.

\subsection{Preliminaries on Markov shifts}
We briefly recall some properties of infinite measure-preserving countable state Markov shifts. Let $S$ be a countable set, which in our case will be $\mathbb Z$, and let $P$ be a stochastic matrix over $S$. Let  $\lambda$ be a vector indexed by $S$ that is a left-eigenvector of $P$ with  eigenvalue $1$, so $\lambda P = \lambda$, and assume that $\sum_{s \in S}\lambda_s=\infty$. Let $X = S^{\Z}$, let $\mathcal{B}$ be the Borel $\sigma$-algebra generated by cylinder  sets of the form
\begin{equation*}
	[s_0 \ldots s_n]_k = \{x \in X \mid x_{j + k} = s_j \text{ for all } k = 0, \ldots, n \}.
\end{equation*}
Define a measure on these sets by
\begin{equation*}
	\mu_{\lambda}([s_0 \ldots s_n]_k) = \lambda_{s_0} p_{s_0, s_1} p_{s_1, s_2} \ldots p_{s_{n-1}, s_n}
\end{equation*}
and let $T$ be the left shift  on $X$. Then  $T$ preserves $\mu_{\lambda}$. The tuple $(X, \mathcal{B}, \mu, T)$ is called a {\bf $\sigma$-finite Markov shift}.\\

Let $P^n$ be the matrix $P$ taken to the $n$th power, and let $p_{s,t}^{(n)}$ be the $(s,t)$-th entry of $P^n$.	A Markov shift is called {\bf irreducible} if for each $s, t \in S$, we have that $p_{s,t}^{(n)} > 0$ for some $n$. The following can be found in \cite{Aa97}.

\begin{theorem}
	Let $T$ be an irreducible Markov shift. If there is $s \in S$ such that $\sum_{n = 1}^\infty p_{s,s}^{(n)} = \infty$, then $T$ is conservative. Conversely, if there is $s$ such that $\sum_{n = 1}^\infty p_{s,s}^{(n)} < \infty$, then $T$ is not conservative. Furthermore, if $T$ is irreducible and conservative, then it is ergodic.
\end{theorem}

We will use the following theorem of Kakutani and Parry.  

\begin{theorem}[\cite{KaPa63}]
	The following conditions hold if and only if $T^{(k)} = T \times \cdots \times T$ is ergodic:
	\begin{itemize}
		\item [$I_k$.] If $s_1, \ldots, s_k, t_1, \ldots, t_k \in S$, there is $n$ with $p_{s_1, t_1}^{(n)}, \ldots, p_{s_k,t_k}^{(n)} > 0$
		\item [$II_k$.] $\sum_{n=1}^\infty p_{0,0}^{(n)} = \infty$.
	\end{itemize}
\end{theorem}

In \cite{KaPa63}, the authors construct a family of Markov shifts that have ergodic index $k$ as follows. For some $\epsilon > 0$ (the choice of which determines the ergodic index of the shift), they let $p_{i,i + 1} = (1 - \epsilon/i)/2$, $p_{i, i-1} = (1 + \epsilon/i)/2$ if $i \neq 0$, $p_{0,1} = p_{0,-1} = 1/2$, and $p_{i,j} = 0$ if $j \neq i + 1$ and $j \neq i - 1$. They also define, for $i$ positive,
\begin{equation*}
	\lambda_i = \frac{i \cdot \Gamma(1 + \epsilon) \Gamma(i - \epsilon)}{\Gamma(1 - \epsilon) \Gamma(i + 1 + \epsilon)}
\end{equation*}
and define $\lambda_i = 0$ and $\lambda_i = \lambda_{-i}$ if $i < 0$. They note that $\lambda P = \lambda$, and $\sum_{-\infty}^{\infty} \lambda_i = \infty$. Lastly, using a particular $\epsilon = \epsilon(k)$, they show that $Q = P \cdot P$ has ergodic index $k$.

\subsection{Reversible shifts}

\begin{prop}
	Let $T$ be a Markov shift defined by the matrix $P$ with $1$-eigenvalue $\lambda$. If $P$ is reversible, that is, if $P$ satisfies
	\begin{equation}
		\lambda_i p_{i,j} = \lambda_j p_{j,i}
	\end{equation}
	then $T$ is isomorphic to its inverse.
\end{prop}
\begin{proof}
	Define $\phi \colon X \to X$ by $\phi(x)_i = x_{-i}$. Clearly, $T \circ \phi = \phi \circ T^{-1}$. Now, $\phi^{-1}([s_0 \ldots s_n]_k) = \phi([s_0 \ldots s_n]_k) = [s_n \ldots s_0]_l$ where $l$ is some integer. Now,
	\begin{align*}
		\mu_{\lambda} [s_n \ldots s_0]_l &= \lambda_{s_n} p_{s_n, s_{n-1}} \ldots p_{s_1, s_0}\\
		&= p_{s_{n-1}, s_n} \lambda_{s_{n-1}} p_{s_{n-1},s_{n-2}} \ldots p_{s_1, s_0}\\
		&= p_{s_{n-2},s_{n-1}} p_{s_{n-1}, s_n} \lambda_{s_{n-2}} \ldots p_{s_1, s_0}\\
		&= \ldots = p_{s_0, s_1} \ldots p_{s_{n-1},s_n} \lambda_{s_0}\\
		&= \mu_{\lambda} [s_0 \ldots s_n]_k
	\end{align*}
	Thus $\phi$ is a measure isomorphism.
\end{proof}

\begin{prop}
	Let $P$ and $Q$ be reversible stochastic matrices defining Markov shifts, with the same $1$-eigenvector $\lambda$, and where $P$ and $Q$ commute. Then $P \cdot Q$ is reversible.
\end{prop}
\begin{proof}
	By assumption, $\lambda_i p_{i,j} = \lambda_j p_{j,i}$ and $\lambda_i q_{i,j} = \lambda_j q_{j,i}$ for every $i, j$. Now,
	\begin{align*}
		\lambda_i (pq)_{i,j} &= \lambda_i \sum_k p_{i,k}q_{k,j}\\
		&= \sum_k \lambda_i p_{i,k}q_{k,j}\\
		&= \sum_k \lambda_k p_{k,i}q_{k,j}\\
		&= \sum_k \lambda_j p_{k,i}q_{j,k}\\
		&= \lambda_j \sum_k p_{j,k}q_{k,i}\\
		&= \lambda_j (qp)_{j,i}\\
		&= \lambda_j (pq)_{j,i}
	\end{align*}
	so that $P \cdot Q$ is reversible.
\end{proof}

In specific, if $P$ is reversible, then $P \cdot P$ is reversible, because it has the same $1$-eigenvector.

\subsection{Main Construction}

\begin{prop}
	The stochastic matrix $P$ defined by Kakutani and Parry is reversible.
\end{prop}
\begin{proof}
	We wish to show that $\lambda_i/\lambda_j = p_{j,i}/p_{i,j}$. Now,
	\begin{equation*}
		\frac{p_{i,i+1}}{p_{i+1,i}} = \frac{p_{i,i+1}}{p_{i+1,(i+1)-1}} = \frac{1 - \epsilon/i}{1 + \epsilon/(i + 1)}
	\end{equation*}
	so long as $i, i + 1 \neq 0$. If $i = 0$, we have
	\begin{equation*}
		\frac{p_{i,i+1}}{p_{i+1,i}} = \frac{p_{0,1}}{p_{1,0}} = \frac{1}{1 + \epsilon}
	\end{equation*}
	and if $i = -1$, we have
	\begin{equation*}
		\frac{p_{i,i+1}}{p_{i+1,i}} = \frac{p_{-1,0}}{p_{0,-1}} = 1 + \epsilon
	\end{equation*}
	Recall that $\lambda$ is defined as
	\begin{equation*}
		\lambda_i = \frac{i \cdot \Gamma(1 + \epsilon) \Gamma(i - \epsilon)}{\Gamma(1 - \epsilon) \Gamma(i + 1 + \epsilon)}
	\end{equation*}
	if $i > 0$, $\lambda_0 = 1$, and $\lambda_i = \lambda_{-i}$ if $i < 0$. We need only check that the reversibility equality holds if $j = i + 1$ or $i - 1$, as the other entries in $P$ are all zero. If $i > 0$, we have
	\begin{align*}
		\frac{\lambda_{i + 1}}{\lambda_i} &= \frac{(i + 1) \cdot \Gamma(1 + \epsilon) \Gamma(i + 1 - \epsilon)}{\Gamma(1 - \epsilon) \Gamma(i + 2 + \epsilon)} \cdot \frac{\Gamma(1 - \epsilon) \Gamma(i + 1 + \epsilon)}{i \cdot \Gamma(1 + \epsilon) \Gamma(i - \epsilon)}\\
		&= \frac{i + 1}{i} \frac{i - \epsilon}{i + 1 + \epsilon}
	\end{align*}
	whereas
	\begin{align*}
		\frac{p_{i,i+1}}{p_{i+1,i}} = \frac{1 - \epsilon/i}{1 + \epsilon/(i + 1)} = \frac{i + 1}{i} \frac{i - \epsilon}{i + 1 + \epsilon}
	\end{align*}
	which is the same. The $i < -1$ case is a similar calculation. This concludes unless $i = 0, -1$. If $i = 0$, we have
	\begin{equation*}
		\frac{\lambda_1}{\lambda_0} = \frac{\Gamma(1 + \epsilon) \Gamma(1 - \epsilon)}{\Gamma(1 - \epsilon)\Gamma(2 + \epsilon)} = \frac{1}{1 + \epsilon} = \frac{p_{0,1}}{p_{1,0}}
	\end{equation*}
	and if $i = -1$, we get
	\begin{equation*}
		\frac{\lambda_0}{\lambda_{-1}} = \frac{\lambda_0}{\lambda_{1}} = (1 + \epsilon) = \frac{p_{-1,0}}{p_{0,-1}}
	\end{equation*}
	as required.
\end{proof}

\begin{cor}\label{C:tt-1notcons}
	For any $k$, there exists a conservative ergodic Markov shift $T$, isomorphic to its inverse, such that $T^{(k)}$ is conservative ergodic and $T^{(k)} \times T^{-1}$ is neither.
	\end{cor}
\begin{proof}
	Kakutani and Parry show that by suitable choice of $\epsilon$, the Markov shift $T$ defined by $P \cdot P$ is such that $T^{(k)}$ is conservative ergodic but $T^{(k + 1)}$ is not ergodic, hence not conservative. By the above, $T$ is isomorphic to its inverse, so clearly $T^{(k)} \times T^{-1}$ is not conservative (and hence not ergodic). 
\end{proof}
In particular, choosing $k=1$, this gives us a transformation $T$ such that $T$ is conservative ergodic, but $T \times T^{-1}$ is neither.


\subsection{Power Weak Mixing is Generic}

 An invertible transformation $T$ is said to be {\bf power weakly mixing} if for every sequence of numbers $k_1, \ldots, k_r \in \Z \setminus \{0\}$, the product transformation $T^{k_1} \times \ldots \times T^{k_r}$ is ergodic. In finite measure this is equivalent to weak mixing, but in infinite measure it is stronger than infinite ergodic index \cite{AdFrSi01}. As we will show in this section, under the weak topology in the group of invertible measure-preserving transformations, the set of transformations that are power weak mixing is a residual set, so we say this property is {\bf generic}. It follows that the set of transformations $T$ such that $T\times T^{-1}$ is not ergodic is meagre.   Sachdeva \cite{Sa71}  showed that infinite ergodic index is generic in the weak topology. Ageev, at the time of \cite{AgSi02} mentioned to one of the authors that he had a proof that power weak mixing is generic, but it has not  been published as far as we know. Following the proof of genericity of infinite ergodic index in \cite{DaSi09} we include below a proof of genericity of power weak mixing, as we are interested in showing that the properties of the transformations of Section~\ref{S:Tinv} are topologically rare.

We recall the weak topology defined on the group  $\mathcal{G} = \mathcal{G}(X, \mu)$  of invertible measure-preserving transformations on a $\sigma$-finite Lebesgue measure space $(X, \mathcal{B}, \mu)$.   The topology on $\mathcal{G}$ is inherited from  the strong operator topology so that a sequence  $T_n$ converges to $T$ if and only if \[\mu \left( T_n(A) \bigtriangleup T(A) \right) + \mu \left( T_n^{-1}(A) \bigtriangleup T^{-1}(A) \right) \to 0,\]
for all sets of finite measure $A$.  This topology is called the {\bf weak topology} on $\mathcal{G}$, and  is completely metrizable through a natural metric \cite{Sa71}.

We will use  the following lemma from \cite{Sa71}.

\begin{lemma} \label{conjugates_dense}
	The conjugacy class of any transformation $T \in \mathcal{G}(X, \mu)$ is dense in $\mathcal{G}(X, \mu)$.
\end{lemma}

\begin{theorem}
	The property of power weak mixing is generic in $\mathcal{G}(X, \mu)$, in particular, the set of power weakly mixing transformation in $\mathcal{G}(X, \mu)$ forms a dense $G_{\delta}$ subset.
\end{theorem}
\begin{proof}
	Let $P_{\infty}$ be the set of power weakly mixing transformations on $(X, \mu)$. First we show that it is a $G_{\delta}$ set. Let $\alpha = (\alpha_1, \ldots, \alpha_k)$, where $\alpha_i \in \Z \setminus \{0\}$ for each $1\le i\le k$. For  an invertible measure-preserving transformation $T$, define $T^\alpha = T^{\alpha_1} \times \ldots \times T^{\alpha_k}$. That $T$ is power weakly mixing is equivalent to $T^\alpha$ being ergodic for every such $\alpha$. Now, define $\phi_{\alpha} \colon \mathcal{G}(X, \mu) \to \mathcal{G} \left( X^{(k)}, \mu^{(k)} \right)$ by $\phi_{\alpha}(T) = T^\alpha$. As is easily checked, $\phi_{\alpha}$ is continuous in the weak topology. By Sachdeva \cite{Sa71} (see also \cite{Aa97}), the ergodic transformations $\mathcal{E}^{(k)}$ form a $G_{\delta}$ subset of $\mathcal{G}\left(X^{(k)}, \mu^{(k)} \right)$, hence $\phi_{\alpha}^{-1}\left( \mathcal{E}^{(k)} \right)$ is a $G_{\delta}$ subset of $\mathcal{G}(X, \mu)$. But $\phi_{\alpha}^{-1}\left( \mathcal{E}^{(k)} \right)$ is precisely those $T \in \mathcal{G}(X, \mu)$ such that $T^\alpha$ is ergodic, hence $P_{\alpha}$, the set of $T$ such that $T^\alpha$ is ergodic, is $G_{\delta}$. Because the countable intersection of $G_{\delta}$ sets is $G_{\delta}$, $P_{\infty}$ is $G_{\delta}$ in $\mathcal{G}(X, \mu)$.\\

	It remains to show density. Since $P_{\infty}$ is nonempty \cite{DGMS99},  if we show that it is closed under conjugation, Lemma~\ref{conjugates_dense} will give us that it is dense. To that end, let $\alpha = (\alpha_1, \ldots, \alpha_k)$ be a tuple of nonzero integers, let $S$ be a measure-preserving transformation, and suppose that $(S \circ T \circ S^{-1})^\alpha (A) = A$ for some $A$. This means
	\begin{gather*}
		(S \circ T \circ S^{-1})^\alpha (A) = A\\
		\left( (S \circ T \circ S^{-1})^{\alpha_1} \times \ldots \times (S \circ T \circ S^{-1})^{\alpha_k} \right) (A) = A\\
		\left( (S \circ T^{\alpha_1} \circ S^{-1}) \times \ldots \times (S \circ T^{\alpha_k} \circ S^{-1}) \right) (A) = A\\
		S^{(k)} \circ T^\alpha \circ (S^{-1})^{(k)} A = A\\
		T^\alpha \circ (S^{-1})^{(k)} A = (S^{-1})^{(k)} A\\
	\end{gather*}
	hence by the ergodicity of $T^\alpha$ we have $(S^{-1})^{(k)} A$ is either null or conull, hence as $S$ is measure-preserving $A$ is either null or conull, hence $S \circ T \circ S^{-1}$ is power weakly mixing.
\end{proof}

\bibliography{ErgBib_Master}
\bibliographystyle{plain}

\end{document}